\documentclass[a4paper,reqno]{amsart}

\usepackage{amssymb}
\usepackage{amstext}
\usepackage{amsmath}
\usepackage{amscd}
\usepackage{amsthm}
\usepackage{amsfonts}
\usepackage{enumerate}
\usepackage{graphicx}
\usepackage{latexsym}
\usepackage{mathrsfs}
\usepackage[all]{xy}
\xyoption{all}

\usepackage{pstricks}
\usepackage{lscape}
\usepackage{comment}

\newcommand{\old}[1]{{\red #1}}

\newtheorem{theorem}{Theorem}[section]

\newtheorem{corollary}[theorem]{Corollary}
\newtheorem{lemma}[theorem]{Lemma}
\newtheorem{proposition}[theorem]{Proposition}
\newtheorem{definition-proposition}[theorem]{Definition-Proposition}
\newtheorem{question}[theorem]{Question}

\theoremstyle{definition}
\newtheorem{definition}[theorem]{Definition}

\newtheorem{example}[theorem]{Example}

\newcommand{\KG}{\mathop{{\rm K}_0}\nolimits}
\newcommand{\Cl}{\operatorname{Cl}\nolimits}
\newcommand{\CH}{\operatorname{CH}\nolimits}
\renewcommand{\NG}{\operatorname{NG}\nolimits}
\newcommand{\supp}{\operatorname{supp}\nolimits}
\newcommand{\sCM}{\operatorname{\Omega{CM}}}

\newcommand{\mm}{{\mathfrak{m}}}
\newcommand{\pp}{{\mathfrak{p}}}
\newcommand{\CC}{\mathscr{C}}
\newcommand{\DD}{\mathscr{D}}

\newcommand{\Z}{\mathbb{Z}}
\newcommand{\Q}{\mathbb{Q}}

\newcommand{\add}{\operatorname{add}\nolimits}
\newcommand{\fl}{\operatorname{f.l.}\nolimits}

\newcommand{\depth}{\operatorname{depth}\nolimits}
\newcommand{\pd}{\operatorname{proj.dim}\nolimits}
\renewcommand{\mod}{\operatorname{mod}\nolimits}

\newcommand{\Ext}{\operatorname{Ext}\nolimits}
\newcommand{\Hom}{\operatorname{Hom}\nolimits}
\newcommand{\End}{\operatorname{End}\nolimits}
\newcommand{\gl}{\operatorname{gl.dim}\nolimits}

\newcommand{\RHom}{\mathbf{R}\strut\kern-.2em\operatorname{Hom}\nolimits}

\newcommand{\Cok}{\operatorname{Cok}\nolimits}
\newcommand{\Proj}{\operatorname{Proj}\nolimits}

\DeclareMathOperator{\Spec}{Spec}
\newcommand{\ses}[3]{0 \to {#1} \to {#2} \to {#3} \to 0}

\DeclareMathOperator{\C}{\mathbb C}
\DeclareMathOperator{\Sing}{Sing}

\begin{document}

\title{Non-commutative resolutions and Grothendieck groups}
\author{Hailong Dao, Osamu Iyama, Ryo Takahashi and Charles Vial}
\address{H. Dao: Department of Mathematics, University of Kansas, Lawrence, KS 66045-7523, USA}
\email{hdao@math.ku.edu}
\urladdr{http://www.math.ku.edu/~hdao/}
\address{O. Iyama: Graduate School of Mathematics, Nagoya University, Furocho, Chikusaku, Nagoya 464-8602, Japan}
\email{iyama@math.nagoya-u.ac.jp}
\urladdr{http://www.math.nagoya-u.ac.jp/~iyama/}
\address{R. Takahashi: Graduate School of Mathematics, Nagoya University, Furocho, Chikusaku, Nagoya 464-8602, Japan/Department of Mathematics, University of Nebraska, Lincoln, NE 68588-0130, USA}
\email{takahashi@math.nagoya-u.ac.jp}
\urladdr{http://www.math.nagoya-u.ac.jp/~takahashi/}
\address{C. Vial: DPMMS, University of Cambridge, Wilberforce Road,  Cambridge, CB3 0WB, UK}
\email{c.vial@dpmms.cam.ac.uk}
\urladdr{http://www.dpmms.cam.ac.uk/~cv248/}
\thanks{The first author was partially supported by NSF grants DMS 0834050 and DMS 1104017. 
The second author was partially supported by JSPS Grant-in-Aid for Scientific Research 24340004, 23540045, 20244001 and 22224001.
The third author was partially supported by JSPS Grant-in-Aid for Young Scientists (B) 22740008 and by JSPS Postdoctoral Fellowships for Research Abroad. The fourth  {author} was supported by EPSRC Postdoctoral Fellowship EP/H028870/1.}
\begin{abstract}
Let $R$ be a noetherian normal domain. We investigate when $R$ admits a faithful module whose endomorphism ring has finite global dimension. This can be viewed as a non-commutative desingularization of $\Spec(R)$. We show that the existence of such modules forces stringent conditions on the Grothendieck group of finitely generated modules over $R$. In some cases those conditions are enough to imply that $\Spec(R)$ has only rational singularities. 
\end{abstract}
\maketitle

{\section{Introduction}}

The use of non-commutative algebras with finite global dimension was
initiated by M. Auslander in representation theory of Cohen-Macaulay
rings, or more generally, orders \cite{A1,A2,Y}.  He introduced two
important classes of non-commutative algebras with finite global
dimension called {Auslander algebras} and {non-singular
  orders}.  Auslander algebras are defined as the endomorphism
algebras of additive generators in the category of Cohen-Macaulay
modules over representation-finite orders. Then representation theory
of representation-finite orders is encoded in the structure of their
Auslander algebras, and this picture was the starting point of
Auslander-Reiten theory. On the other hand, the representation theory of non-singular
orders is most basic since all Cohen-Macaulay modules are
projective.

Recently, the study of such algebras have found striking applications
in algebraic geometry. Perhaps the most well-known example is Van den
Bergh's definition of {\it non-commutative crepant resolution},
usually abbreviated by NCCR (\cite{VdB}). We recall Van den Bergh's
definition, in a slightly more general setting: { A reflexive module
  $M$ over a commutative noetherian normal domain $R$ is said to give
  an NCCR of $\Spec(R)$ if $\Lambda=\End_R(M)$ is a non-singular
  $R$-order (i.e. $\Lambda_\pp$ is a maximal Cohen-Macaulay
  $R_\pp$-module and $\gl\Lambda_\pp=\dim R_\pp$ for any
  $\pp\in\Spec(R)$).  }

These non-commutative objects provide a particularly pleasant
explanation of the Bondal-Orlov conjecture on the derived equivalence of 
threefolds related by a flop. Van den Bergh's
work has quickly generated a sizable body of research, see for example
the recent survey \cite{L}.

In this note we study a weaker notion which is called a
\emph{non-commutative resolution} (\emph{NCR}). \\

\begin{definition}\label{ncr} {A finitely generated module $M$ over a commutative
    noetherian ring $R$} is said to give a NCR of $\Spec(R)$ (or just $R$ by abuse of notation) if $M$
  is faithful and $\End_R(M)$ has finite global dimension.
\end{definition}

The assumption that $M$ is faithful in Definition \ref{ncr} is reasonable since,
for example, {any simple $R$-module $k$} \old{} gives  $k=\End_R(k)$ which  always  has finite global dimension.
Basic examples of NCRs appeared in representation theory: Auslander algebras \cite{A1},
higher Auslander algebras \cite{I2,IW2} and NCCRs \cite{VdB}. It is known that NCRs exist when $R$ is artinian, or reduced and one-dimensional (see \cite{L}). 

We shall try to give information on the following
\begin{question}
When does $R$ have a NCR?
\end{question}

We give necessary conditions which 
focus on the {\it Grothendieck group} of the category of finitely
generated modules over $R$ and its subcategories (see 
\ref{application1}, \ref{application2}, \ref{main}). Our results show
that the existence of NCRs still implies strong constraints on the
singularities of $R$. For example, we prove that a standard graded
Cohen-Macaulay algebra $R$ over $\C$ with only rational singularities
outside the irrelevant ideal has a NCR only if $R$ has rational
singularities (Theorem \ref{S3}). Our proofs utilize some non-trivial
results from algebraic K-theory.  For surface singularities over an
algebraically closed field, we observe that the existence of NCRs
actually characterizes rational singularities (Corollary \ref{rat}).
This adds a new member to a long list of interesting equivalent
conditions for rationality of surface singularities.


\section{NCRs and Grothendieck groups}

Throughout this note, we denote by $R$ a commutative noetherian ring, and by $\mod R$ the category of finitely generated $R$-modules. We denote by $\KG(R)$ the Grothendieck group of the abelian category $\mod R$.
When $R$ is a normal domain, we denote by $\Cl(R)$ the class group of $R$.
Our main technical result (Theorem \ref{main}) given later in this section relates
the existence of NCRs and certain finiteness properties of Grothendieck groups and class groups.
The first application is the following,
which deals with NCRs  over a normal domain.

\begin{corollary}\label{application1}
Let $R$ be a semilocal normal domain. If $R$ has a NCR, then $\Cl(R)$ is a finitely generated abelian group.
\end{corollary}

A typical example of a ring which has an infinitely generated class group is
$R:=\C[[x,y,z]]/(x^3+y^3+z^3)$, see \ref{Chowfacts}.
In particular, $R$ has no faithful module giving rise to a NCR by the above theorem.
This  {gives} an answer to a question by Burban \cite{B}.

The second application of our main result is the following,
which deals with NCRs given by modules which are  locally {generators} on the outside of a closed set of {dimension at most one}.

\begin{corollary}\label{application2}
Let {$R$} be a {semilocal} ring. If $M$ is an $R$-module which is locally a generator outside a closed subscheme of  $\Spec(R)$ of {dimension at most one} and gives {a} NCR, then $\KG(R)$ is a finitely generated abelian group.
\end{corollary}

Let us now state and prove our key technical result.
Let $(-)^*:=\Hom_R(-,R)$.
For an $R$-module $M$, let $E:=\End_R(M)$ and we denote by
\[a_M:M\otimes_EM^*\to R\]
the natural map sending $m\otimes f\in M\otimes_EM^*$ to $f(m)$.

\begin{definition}
We define the \emph{non-generating locus} $\NG(M)$ of $M$ as the support of $\Cok(a_M)$. 
\end{definition}

{For $M\in\mod R$, we denote by $\add_RM=\add M$ the full subcategory of $\mod R$ consisting of direct summands of direct sums of copies of $M$.}
Recall that $M\in\mod R$ is called a \emph{generator} of $\mod R$ if there exists
a surjection from a direct sum of copies of $M$ to $R$, or
equivalently, $R\in\add M$.
The following observation explains the name of $\NG(M)$.

\begin{proposition}
$\NG(M)$ is the set of prime ideals $\pp$ such that
$M_{\pp}$ is not a generator of {$\mod R_\pp$}.
\end{proposition}

\begin{proof}
Let $\pp$ be a prime ideal of $R$.
Then $M_{\pp}$ is a generator of {$\mod R_\pp$}, if and only if $\sum_{f\in M_{\pp}^{*}}f(M_{\pp})=R_{\pp}$, if and only if $(a_M)_{\pp}$ is surjective, if and only if $(\Cok a_M)_{\pp}=0$.
\end{proof}

For a {full} subcategory $\CC$ of $\mod R$, we denote by $\langle\CC\rangle$
the subgroup of $\KG(R)$ generated by elements $[X]$ with $X\in\CC$.
Our main result in this section is the following:

\begin{theorem}\label{main}
Let $R$ be a semilocal ring and $M$ give a NCR of $R$.
Let $\CC_{M}$ be the {full} subcategory of $\mod R$ consisting of $X$ satisfying
$\supp X\subset\NG(M)$.
Then $\KG(R)/\langle\CC_{M}\rangle$ is a finitely generated abelian group.
\end{theorem}

First let us recall a well-known fact on Grothendieck groups.
Let $\CC$ be a Serre subcategory of $\mod R$ (i.e. $\CC$ is
closed under submodules, factor modules and extensions), and
let $(\mod R)/\CC$ be the quotient abelian category of $\mod R$
\cite{Pop}:
The objects of $(\mod R)/\CC$ is the same as $\mod R$, and the morphism set is given by
\[\Hom_{(\mod R)/\CC}(X,Y):=\varinjlim_{X',Y'}\Hom_R(X',Y/Y')\]
where $X'$ and $Y'$ run over all submodules of $X$ and $Y$ respectively such that $X/X',Y'\in\CC$.
In this case we have the following observation.

\begin{proposition}\label{quotient Grothendieck}\cite{He}
$\KG((\mod R)/\CC)$ is isomorphic to $\KG(R)/\langle\CC\rangle$.
\end{proposition}

We need the following general observations on generators.

\begin{lemma}\label{property of generator}
Let $M\in\mod R$ be a generator and $E:=\End_R(M)$.
Then we have the following properties.
\begin{itemize}
\item[(a)] $M^*$ is a projective $E$-module.
\item[(b)] The natural map $a_M:M\otimes_EM^*\to R$ is an isomorphism.
\end{itemize}
\end{lemma}

\begin{proof}
(a) Since $R\in\add_RM$, we have $M^*\in\add_E\Hom_R(M,M)=\add_EE$.

(b) {For any $X\in\mod R$, we denote by $b_X:M\otimes_E\Hom_R(M,X)\to X$ the natural map sending $m\otimes f$ to $f(m)$.
This gives a natural transformation $b:M\otimes_E\Hom_R(M,-)\to{\rm 1_{\mod R}}$ of additive functors $\mod R\to\mod R$.
Since $b_M$ is clearly an isomorphism, so is $b_X$ for any $X\in\add_RM$.
In particular, $b_R$ is an isomorphism. Since $a_M=b_R$, we have the assertion.}
\end{proof}

Clearly $\CC_{M}$ is a Serre subcategory of $\mod R$. 
Let $(\mod R)/\CC_{M}$ be the quotient abelian category of $\mod R$.
We define a functor
\[F:\mod E\xrightarrow{\Hom_E(M^*,-)}\mod R\to(\mod R)/\CC_{M}\]
where $\mod R\to(\mod R)/\CC_{M}$ is a natural functor.

\begin{lemma}\label{exactness}
$F$ is an exact functor.
\end{lemma}

\begin{proof}
Let $0\to X\to Y\to Z\to 0$ be an exact sequence in $\mod E$.
Applying $\Hom_E(M^*,-)$, we have an exact sequence
\[0\to\Hom_E(M^*,X)\to\Hom_E(M^*,Y)\to\Hom_E(M^*,Z)\to\Ext^1_E(M^*,X)\]
We only have to show $\Ext^1_E(M^*,X)\in\CC_{M}$.
For any prime ideal $\pp\notin\NG(M)$, we have that $M_{\pp}^*$ is a projective
$E_{\pp}$-module by Lemma \ref{property of generator}. Thus we have
\[\Ext^1_E(M^*,X)_{\pp}=\Ext^1_{E_{\pp}}(M_{\pp}^*,X_{\pp})=0,\]
and so $\supp\Ext^1_E(M^*,X)\subset\NG(M)$.
\end{proof}

Next we show the following property of $F$.

\begin{lemma}\label{density}
$F$ is a dense functor.
\end{lemma}

\begin{proof}
For any $X\in\mod R$, let $Y:=\Hom_R(M,X)\in\mod E$.
Then we have
\[F(Y)=\Hom_E(M^*,\Hom_R(M,X))\cong\Hom_R(M\otimes_EM^*,X).\]
For any prime ideal $\pp\notin\NG(M)$, 
we have that $M_{\pp}$ is a generator of  {$\mod R_\pp$}. Thus $(a_M)_{\pp} {:(M\otimes_EM^*)_{\pp}\to R_{\pp}}$ is an isomorphism by Lemma \ref{property of generator}(b).
Hence the natural map
\[(a_M\cdot):X\to\Hom_R(M\otimes_EM^*,X)=F(Y)\]
induced by $a_M$ has the kernel and the cokernel in $\CC_{M}$.
Consequently $X$ is isomorphic to $F(Y)$ in $(\mod R)/\CC_{M}$.
\end{proof}

By Lemma \ref{exactness} and Proposition \ref{quotient Grothendieck},
we have a homomorphism
\[\KG(E)\to \KG((\mod R)/\CC_{M})\cong \KG(R)/\langle\CC_{M}\rangle\]
of abelian groups. This is surjective by Lemma \ref{density}.

\begin{lemma}
Let $R$ be a semilocal ring and $E$ a module-finite $R$-algebra.
If the global dimension of $E$ is finite, then $\KG(E)$ is finitely generated.
\end{lemma}

\begin{proof}
Since the global dimension is finite, $\KG(E)$ is generated by indecomposable
projective $E$-modules.
Since $R$ is semilocal, it follows from \cite[Theorem 9]{FS} that there exist only
finitely many isomorphism classes of indecomposable projective $E$-modules.
Thus $\KG(E)$ is finitely generated.
\end{proof}

The above lemma completes the proof of Theorem \ref{main}.
\qed

\medskip
Now we prove Corollary \ref{application1}.
We need the following fact, see \cite{Ch}.

\begin{proposition}\label{Grothendieck to class}
Let $R$ be a normal domain and $\Phi$ be the set of prime ideals of $R$ with height at least two.
Let $\DD$ be the {full} subcategory of $\mod R$ consisting of $X$ satisfying
$\supp X\subset\Phi$. Then $\KG(R)/\langle\DD\rangle$ is isomorphic to $\mathbb Z \oplus \Cl(R)$.
\end{proposition}

{Assume that $M$ gives a NCR of $R$.}
Since $R$ is a normal domain and $M$ is a faithful $R$-module,
we have that $M_{\pp}$ is a faithful $R_{\pp}$-module 
and hence $M_{\pp}$ has a non-zero free summand for any prime ideal $\pp\notin\Phi$.
Thus we have $\NG(M)\subset\Phi$ and $\CC_{M}\subset\DD$.
By Theorem \ref{main}, we have that $\KG(R)/\langle\DD\rangle$ is finitely generated.
By Proposition \ref{Grothendieck to class}, we have the assertion.
\qed

\medskip
Finally we prove Corollary \ref{application2}.
As $R$ is semilocal and $\dim \NG(M)\ge 1$,
the set $\NG(M)$ is finite. Since $\langle\CC_M\rangle$ is generated by $R/\pp$ with
$\pp\in \NG(M)$, it is finitely generated. Since $\KG(R)/\langle\CC_M\rangle$ is finitely
generated by Theorem \ref{main}, so is $\KG(R)$.

\qed

%
%

\section{NCRs and rational singularities}

In this section let $R$ be a normal domain containing a field $k$. We wish to discuss the following:

\begin{question}\label{ratques}
Suppose $R$ has a NCR. When can we deduce that $\Spec(R)$ has rational singularities? 
\end{question}

Recall that a variety $Y$  is said to have {\it rational singularities} if for any (equivalently, some) resolution of singularity $f: X \to Y$, we have  $R^if_*\mathcal O_X = 0 $ for $i>0$. When $Y = \Spec(R)$ this reduces to $H^i(X, \mathcal O_X)=0$ for $i>0$ (see \cite[Section 1]{Sm}).

The above question is motivated by a beautiful result by  Stafford and Van den Bergh (\cite[Theorem 4.2]{SV}):

\begin{theorem} (Stafford-Van den Bergh)
Let  $k$ be an algebraically closed field of characteristic $0$ and $\Delta$ be a prime affine $k$-algebra that is finitely generated as a module over its center $Z(\Delta)$.
If $\Delta$ is a non-singular $Z(\Delta)$-order  then $Z(\Delta)$ has only rational singularities. 

In particular, suppose $R$ is a Gorenstein normal affine $k$-algebra. If $R$ has an NCCR, then $\Spec(R)$ has only rational singularities.  
\end{theorem} 

In fact, at the end of their paper Stafford and Van den Bergh raised the question of whether it is enough to only assume that we have {a} maximal Cohen-Macaulay {module giving a} NCR but $R$ is not necessarily Gorenstein (\cite[Question 5.2]{SV}). 

Our first result shows that having rational (isolated) singularity characterizes the existence of NCRs for surface
singularities.  For a Cohen-Macaulay ring $R$, let $\sCM(R)$ denote
the category of first syzygies of some maximal Cohen-Macaulay modules.
{Thus $\sCM(R)$ consists of all $X\in\mod R$ such that there exists an
  exact sequence $0\to X\to P\to Y\to0$ with a projective $R$-module
  $P$ and a maximal Cohen-Macaulay $R$-module $Y$.}

\begin{corollary}\label{rat}
  Let $(R,\mm,k)$ be a local normal domain of dimension two. Consider
  the following:
\begin{enumerate}[\rm(1)]
\item $\sCM(R)$ is of finite type (that is, there exists $M \in
  \mod(R)$ such that $\sCM(R)=\add M$).
\item $R$ has a NCR.
\item $\Cl(R)$ is a finitely generated abelian group.
\item $\KG( R)\otimes_{\Z}\Q$ is a finite dimensional $\Q$-vector
  space (equivalently, $\Cl(R)$ has a finite rank).
\item $\Spec(R)$ has rational singularities.
\end{enumerate}
Then $(1) \Rightarrow (2) \Rightarrow (3) \Rightarrow (4)$. If $R$ is
excellent, henselian and $k$ is algebraically closed, then $(3)
\Rightarrow {(5)}\Rightarrow (1)$.  If in addition $k$ has
characteristic $0$ then $(4) \Rightarrow (5)$.
\end{corollary}

\begin{proof}
  The implication $(1)\Rightarrow(2)$ is \cite[Theorem 2.10]{IW1} and
  $(2)\Rightarrow(3)$ is Corollary \ref{application1}. The implication
  $(3)\Rightarrow(4)$ is trivial.  The implication $(3) \Rightarrow
  (5)$ is essentially \cite[17.3]{Lip}. The proof works almost the
  same, except for the crucial Complement 11.3, where one needs to
  replace condition (2) of Theorem 1.7 in \cite{Ar} by the condition
  that the Picard group of the curve $Z$ is finitely generated.
  Similarly $(4) \Rightarrow (5)$ in characteristic $0$.  The
  statement $(5)\Rightarrow(1)$ follows from the fact that there are
  only finitely many indecomposable special Cohen-Macaulay modules
  (see \cite{Wu}, \cite[Theorems 3.6 and 2.10]{IW1}). Note that the
  result in \cite{Wu} was stated for singularities over complex
  numbers, but the proof also works for our case, the extra
  information we need is the existence of a (minimal)
  desingularization of $\Spec R$, which is known (cf. \cite[Theorem
  4.1]{Lip}). One can bypass the use of  Grauert-Riemenschneider vanishing used in Wunram's proof by the 
  discussion before Theorem 5 in \cite{Gav}.
\end{proof}

\begin{example}
  The implication $(3) \Rightarrow (5)$ in Corollary \ref{rat} really
  requires all the assumptions.  It is not  true when $k$ is not algebraically closed (but
  $R$ is complete): Salmon (\cite{Sal}) showed that
  $k(u)[[x,y,z]]/(x^2+y^3+uz^6)$ is factorial for any field $k$.  The
  condition that $R$ is henselian is also crucial: the ring
  $R=k[x,y,z]_{(x,y,z)}/(x^r+y^s+z^t)$ where $r,s,t$ are pairwise
  prime, is factorial over any field $k$ (\cite[Corollary 10.17]{F}).

  Also, the implication $(4) \Rightarrow (5)$ may fail in positive
  characteristics.  In fact, when $k = \overline {\mathbb F_p}$ the
  class group will always be locally finite.

\end{example}

Now we discuss Question \ref{ratques} in higher dimension. 
{We may assume $R$ is a complete local ring to study Question \ref{ratques} by the following:}

\begin{lemma}\label{locom}
Suppose $M$ gives a NCR {of} $R$ and $\pp \in \Spec(R)$. Then $M_{\pp}$ gives a NCR {of} $R_{\pp}$. If $(R, \mathfrak m)$ is local then the completion $\hat M$ of $M$ gives a NCR {of} $\hat R$.
\end{lemma}

\begin{proof}
Let $\Lambda$ be a module-finite $R$-algebra. 
{We denote by $\fl\Lambda$ the category of $\Lambda$-modules of finite length.}
We only have to show $\gl(S\otimes_R\Lambda)\le\gl\Lambda$ for $S:=R_{\pp}$ or $S:=\hat R$.

When $S=R_\pp$ (respectively, $S=\hat R$), there is an exact dense functor $S\otimes_R-:\mod\Lambda\to\mod(S\otimes_R\Lambda)$ (respectively, $S\otimes_R-:\fl\Lambda\to\fl(S\otimes_R\Lambda)$).
In particular we have $\pd_{S\otimes_R\Lambda}(S\otimes_RM)\le\pd_{\Lambda}M$.
Since $\gl\Lambda=\sup_{X\in\mod\Lambda}\{\pd_\Lambda X\}=\sup_{X\in\fl\Lambda}\{\pd_\Lambda X\}$, the assertion follows.
\end{proof}

The next result shows that NCRs also behave well under separable field
extensions.

\begin{lemma}\label{field}
Let $R$ be a commutative {algebra over a field $K$} and $L$ a separable {field} extension of $K$.
If $M$ gives a NCR of $R$, then $L\otimes_KM$ gives a NCR of $L\otimes_KR$.
\end{lemma}

\begin{proof}
Let $E:=\End_R(M)$.
Clearly we have {$\End_{L\otimes_KR}(L\otimes_KM)=L\otimes_KE$}.
We only have to show that $L\otimes_KE$ has finite global dimension.
For any $X\in\mod E$, clearly the $L\otimes_KE$-module $L\otimes_KX$ has 
finite projective dimension.
It is enough to show that any simple $L\otimes_KE$-module $S$ is a direct summand of $L\otimes_KX$ for some $X\in\mod E$.

We regard $S$ as an $E$-module, and we take a simple {$E$-submodule} $X$ of $S$.
Then we have $LX=S$, and we have a surjection $L\otimes_KX\to S$ of
$L\otimes_KE$-modules sending $l\otimes x\to lx$.
Since $L$ is a separable extension of $K$, we have that $L\otimes_KX$ is a
semisimple $L\otimes_KE$-module \cite[Corollary 7.8(ii)]{CR}. Thus $S$ is a direct summand of $L\otimes_KX$ and we complete the proof.
\end{proof}

For a scheme $X$ let $\CH_i(X)$ denote the Chow group of algebraic
cycles of dimension $i$ and $\CH_*(X)$ the total Chow group. We shall
need the following well-known facts, see \cite[Exercise II.6.3]{Ha2},
\cite{Fu} and \cite{Ku1}:

\begin{theorem}\label{Chowfacts}
  Let $A$ be the {homogeneous} coordinate ring of a projective variety
  $X$ over a field $k$ and $R$ be the local ring of $A$ at the
  irrelevant ideal. Let $h$ denote the class in $\CH(X)_{\mathbb Q}$ of
  a hyperplane section on $X$.
\begin{enumerate}[\rm(1)]
\item There is an exact sequence
$$\ses{\mathbb Z}{\Cl(X)}{\Cl(R)}  $$
where the first map sends $1$ to $h$.
\item 
  We have the following isomorphisms of $\mathbb Q$-vector spaces:
$$\CH_*(X)_{\mathbb Q}/(h\cap \CH_*(X)_{\mathbb Q}) \cong \CH_*(A)_{\Q}
\cong \CH_*(R)_{\mathbb Q} \cong \KG(R)_{\mathbb Q} $$ where the first
two are graded isomorphisms.
\end{enumerate}
\end{theorem}

The second isomorphism in Theorem \ref{Chowfacts}(2) is \cite[Lemma
4.1]{Ku1} and the first isomorphism is only stated for $X$ smooth in
\cite[Theorem 1.3]{Ku1}. However we notice that the isomorphism holds
without assuming $X$ smooth, \emph{cf.} the proof of Proposition
\ref{support}. \medskip


Before moving on  we
recall  the definition of Serre's conditions $(S_n)$. For a non-negative integer $n$, $M$ is said to satisfy $(S_n)$ if:
$$ \depth_{R_p}M_p \geq \min\{n,\dim(R_p)\} \ \forall p\in \Spec(R)$$

\begin{proposition}\label{aus}
  Let $R$ be a normal local ring. Let $M$ be a finitely generated
  faithful $R$-module. {Then $\NG(M)$ is a closed subscheme of
    $\Spec(R)$ of codimension at least $2$ .  If in addition we assume}
  that $\End_R(M)$ is $(S_3)$, then $M$ is locally free outside  the
  singular locus $\Sing(R)$. In particular, $\NG(M) \subseteq
  \Sing(R)$.

  \noindent If moreover $M$ gives a NCR and $\dim \Sing(R) \leq 1$,
  then $\KG(R)$ is a finitely generated abelian group.
\end{proposition}

\begin{proof}
  The first assertion follows from the proof of \ref{application1}. The last
  assertion follows from Corollary \ref{application2}.  {Assume now
    that $\End_R(M)$ is $(S_3)$.}  The assumption implies that
  $\End_R(M)$ is reflexive as an $R$-module.  Thus $\End_R(M^{**})
  \cong \End_R(M)$, and we may assume that $M$ is reflexive.  Now what
  we need to prove follows from the fact that if $R$ is a regular
  local ring and $M$ is a reflexive $R$-module, then $\End_R(M)$ is
  $(S_3)$ if and only if $M$ is free (see \cite[Corollary 2.9]{HW2}).
\end{proof}

The following result is essentially due to Roitman \cite{Roit}. We
give a proof for sake of completeness. Given a variety $X$,
$\CH_0(X)_{\mathbb Q}$ is said to be \emph{supported} in dimension $l$
if there exists an $l$-dimensional closed subvariety $Z$ of $X$ such
that the proper pushforward map $\CH_0(Z)_{\mathbb Q} \rightarrow
\CH_0(X)_{\mathbb Q}$ is surjective.

\begin{lemma}\label{Chowlem} 
  Let $X$ be a smooth projective variety over $\C$.  If
  $\CH_0(X)_{\mathbb Q}$ is supported in dimension $l$, then
  $H^i(X,\mathcal O_X) =0$ for all $i > l$.
\end{lemma}

\begin{proof}
  By localization there is an $l$-dimensional subscheme $j : Z
  \hookrightarrow X$ such that $\CH_0(X-Z)_{\Q} = 0$. Let $d$ be the
  dimension of $X$ over $\C$. By a result of Bloch-Srinivas on the
  decomposition of the diagonal \cite[Proposition 1]{BS}, there is a
  decomposition $\Delta_X = \Gamma_1 + \Gamma_2 \in \CH_d(X \times
  X)_{\Q}$. Here $\Delta_X$ is the class of the diagonal inside
  $\CH_d(X \times X)_\Q$, $\Gamma_1$ is a cycle supported on $X \times
  Z$ and $\Gamma_2$ is a cycle supported on $D \times X$ for some
  divisor $D$ inside $X$. Let's write $f : \widetilde{D} \rightarrow
  D$ for a resolution of singularities of $D$ and $g : \widetilde{Z}
  \rightarrow Z$ for a resolution of singularities of $Z$. Then the
  contravariant actions of $\Gamma_1$ and $\Gamma_2$ on $H^i(X,\Q)$
  are morphisms of Hodge structures. The morphism $\Gamma_1^* :
  H^i(X,\Q) \rightarrow H^i(X,\Q)$ factors through $j^* : H^i(X,\Q)
  \rightarrow H^i(\widetilde{Z},\Q)$ and the morphism $\Gamma_2^* :
  H^i(X,\Q) \rightarrow H^i(X,\Q)$ factors through the Gysin morphism
  $f_* : H^{i-2}(\widetilde{D},\Q) \rightarrow H^i(X,\Q)$. In
  particular, $\Gamma_1$ acts trivially on $H^i(X,\mathcal O_X)$ for
  $i>l$. Now, since $f_*$ is a morphism of Hodge structures of
  bidegree $(1,1)$, it follows that the intersection of the image of
  $f_*$ with $H^i(X,\mathcal O_X) = H^{0,i}(X)$ is zero. Thus, if
  $i>l$ and if $\alpha$ is any cohomology class in $H^i(X,\mathcal
  O_X)$, we have
  $$\alpha = \Delta_X^*\alpha = \Gamma_1^*\alpha + \Gamma_2^*\alpha =
  0,$$ i.e.  $H^i(X,\mathcal O_X)=0$.
\end{proof}

\begin{proposition} \label{support} Let $k$ be a field, let $R$ be a
  standard graded algebra over $k$, i.e. a graded Noetherian ring with
  $R_0=k$ and $R = R_0[R_1]$, of dimension at least $3$. Let
  $\mathfrak{m} := \bigoplus_{i>0} R_i$, $X = \Proj R$ and let $Z$ be
  a closed subscheme of codimension at least $2$ in $\Spec R$. Let
  $\CC$ be the subcategory of $\mod R$ generated by the finitely
  generated $R$-modules $M$ with $\supp M \subseteq Z$.  Assume
  $K_0(R)/\langle \CC \rangle$ is finitely generated. Then
  $\CH_0(X)_\Q$ is supported in codimension $1$.
\end{proposition}
\begin{proof}
  By Riemann-Roch \cite[\S 18]{Fu}, we have an isomorphism $\tau_X :
  \KG(R)_{\mathbb Q} \rightarrow \CH_*(R)_{\mathbb Q}$ which is
  covariant with respect to proper morphisms. The subgroup $\langle
  \CC \rangle$ of $\KG(R)$ is included in the image of $\KG(Z)$ inside
  $\KG(R)$ and it follows that $\KG(R)$ is generated by $\KG(Z)$ via
  the natural inclusion $Z \hookrightarrow \Spec R$ and by finitely
  many classes. Thus $\CH_1(R)_{\mathbb Q}$ is generated by
  $\CH_1(Z)_{\mathbb Q}$ and by finitely many $1$-cycles.  Up to adding
  finitely many components of codimension $\geq 2$ in $\Spec R$ to
  $Z$, we may even assume that $\CH_1(R)_{\mathbb Q}$ is supported on
  $Z$.

  In the proof of \cite[Theorem 1.3]{Ku1}, Kurano establishes the
  existence of the following exact sequence for $v>0$ (see
  \cite[(4.16)]{Ku1} and notice that no smoothness assumption on $X$
  is necessary for (4.16) to hold):
  $$\CH_v(X) \rightarrow \CH_{v-1}(X) \rightarrow \CH_{v}(R) \rightarrow
  0.$$ The map on the left is given by intersecting with
  $c_1({\mathcal O}_X(1))$ and the second map is the composite
  $$\CH_{v-1}(X) \stackrel{\eta^*}{\longrightarrow}
  \CH_v(\widetilde{X}-\{t\}) \stackrel{(j^*)^{-1}}{\longrightarrow}
  \CH_v(\widetilde{X}) \stackrel{k^*}{\longrightarrow} \CH_v(R).$$ Here
  $\widetilde{X}$ is the projective cone over $X$, $\{t\}$ is the
  vertex of $\widetilde{X}$ and $\Spec R = \widetilde{X} - X$ is the
  affine cone over $X$. We refer to \cite{Ku1} for more details.
  Important to us is that $k : \Spec R \rightarrow \widetilde{X}$ and
  $j : \widetilde{X}-\{t\} \rightarrow \widetilde{X}$ are open
  immersions and that $\eta : \widetilde{X}-\{t\} \rightarrow X$ is a
  smooth $\mathbb{A}_k^1$-bundle. In particular these three morphisms
  are flat. Let $\widetilde{Z}$ be the closure of $Z$ inside
  $\widetilde{X}$ and let $Y$ be the image (closed by definition) in
  $X$ of $\widetilde{Z}|_{\widetilde{X}-\{t\}}$ via $\eta$.  By
  definition of flat pullbacks for Chow groups, we see that if
  $\CH_1(R)_{\mathbb Q}$ is supported on $Z$, then the composite map
  $\CH_0(Y)_{\mathbb Q} \rightarrow \CH_0(X)_{\mathbb Q}
  \stackrel{k^*(j^*)^{-1}\eta^*}{\longrightarrow} \CH_1(R)_{\mathbb Q}$
  is surjective. It follows from the short exact sequence above that
  $\CH_0(X)_{\mathbb Q}$ is supported on the union of $Y$ with a
  hyperplane section. It is obvious that each component of $Y$ has
  codimension at least one inside $X$. Therefore $\CH_0(X)_{\mathbb Q}$
  is supported in codimension one.
\end{proof}

Combining Proposition \ref{aus} and Lemma \ref{Chowlem} with Theorem
\ref{main} and Proposition \ref{support}, we obtain:

\begin{theorem}\label{S3} 
  Let $R$ be a normal, Cohen-Macaulay standard graded algebra over a subfield $k$ of 
  $\C$. Let $\mathfrak{m}$ be the
    irrelevant ideal of $R$. Suppose that $\Spec(R)-\{\mathfrak{m}\}$ has only
    rational singularities. Suppose moreover that there exists an
  $R$-module $M$ giving a NCR. 
 Then $\Spec(R)$ has only rational singularities.
\end{theorem}

\begin{proof}
  By Lemma \ref{field} we can assume $k=\C$. Let $X = \Proj R$ and $d=\dim X$. By \cite[Theorem 2.2]{Wan2} we
  only need to show that $H^d(X,\mathcal O_X(n))=0$ for $n\geq 0$. It
  is actually enough to show that $H^d(X,\mathcal O_X)=0$. Indeed,
  letting $H$ be a hyperplane section of $X$, we have a short exact
  sequence for any $i$: $$\ses{\mathcal O_X(i-1)}{\mathcal
    O_X(i)}{\mathcal O_H(i)},$$ whose long exact sequence of
  cohomology gives exact sequences
$$H^d(X,\mathcal O_X(i-1))  \to H^d(X,\mathcal O_X(i))
\to H^d(X,\mathcal O_H(i)) =0 $$ as $\dim H = d-1$. Induction on $n$
shows that $H^d(X,\mathcal O_X(n))=0$ for $n\geq 0$, as desired.

Let's therefore show that $H^d(X, \mathcal O_X)=0$. By Proposition
\ref{aus}, $NG(M)$ is contained in a closed subscheme of $\Spec R$ of
codimension $2$ by normality of $R$. By Theorem \ref{main} and
Proposition \ref{support}, it follows that $\CH_0(X)_{\mathbb Q}$ is
supported on a divisor $D$, i.e.  $\CH_0(X-D)_{\mathbb Q} = 0$. Let $f
: \widetilde{X} \rightarrow X$ be a resolution of singularities of $X$
and let $\widetilde{D} = f^{-1}(D)$. Up to adding some components to
$D$, we may assume that $f$ induces an isomorphism
$\widetilde{X}-\widetilde{D} \rightarrow X-D$. Also we still have
$\CH_0(X-D)_{\mathbb Q} = 0$ and it follows that
$\CH_0(\widetilde{X}-\widetilde{D})_{\mathbb Q} = 0$, i.e. that
$\CH_0(\widetilde{X})_{\mathbb Q}$ is supported on a divisor.  By
Lemma \ref{Chowlem}, we then have $H^d(\widetilde{X}, \mathcal
O_{\widetilde{X}})=0$. Since $X$ has only rational singularities, we
see from the Leray-Serre spectral sequence that $H^d(X, \mathcal
O_X)=0$.
\end{proof}


A consequence of the first half of the proof of Theorem \ref{S3} and of
Lemma \ref{Chowlem} (applied in the case $l=0$) is the following

\begin{corollary}
  Let $X$ be a smooth projective variety over $\C$.  If $\dim_{\Q}
  \CH_0(X)_{\mathbb Q} <\infty$, then $X$ admits an embedding into a
  projective space whose homogeneous coordinate ring has only rational
  singularities.
\end{corollary}


In view of  Proposition \ref{aus} and Theorem \ref{Chowfacts} we ask:
\begin{question}\label{Grorat}
  Let $(R,\mm,k)$ be a Cohen-Macaulay complete local normal domain with $k$ an
  algebraically closed field. If $\KG( R)$ is finitely generated, must
  $\Spec(R)$ have only rational singularities?
\end{question}

Corollary \ref{rat} shows that the answer is yes in dimension $2$. Of
course, in higher dimensions the existence of desingularizations is
not known for positive or mixed characteristics, so one may need to
replace the condition of rational singularities with suitable concepts
such as being $F$-rational or pseudo-rational.

Our last example illustrates some subtlety involving Lemma \ref{Chowlem}.

\begin{example}\label{exCharles}
  Lemma \ref{Chowlem} might not be true over fields of characteristic
  $0$ whose transcendance degree over their prime subfield is not
  large enough. Indeed, consider a K3-surface $X$ over the algebraic
  closure of $\Q$. Then, $H^2(X,\mathcal O_X)$ is not zero but it is
  expected (as part of the Bloch-Beilinson conjectures) that
  $\CH_0(X)_{\Q}=\Q$.

  Provided such an example exists, it could potentially yield a
  negative example to Question \ref{Grorat}. Note that $\CH(X)_{\Q}$
  is a finite dimensional $\Q$-vector space (it is known that the rank
  of the Picard group of $X$ is finite). We can use an very ample line
  bundle on $X$ to embed $X$ into some projective space. Let $R$ be
  the homogeneous coordinate ring of such an embedding, thus
  $\dim_{\Q}\CH(R)_{\Q}<\infty$. But if $R$ has rational
  singularities, then $H^2(X, \mathcal O_X)=0$ (cf.  \cite[Theorem
  2.2]{Wan2} and \cite{Wan1}).
\end{example}

\section*{Acknowledgments}

The first author thanks Joseph Lipman and Karl Schwede for some helpful conversations on rational singularities. The second author thanks Michael Wemyss for valuable discussions.
{This work started when the second and third authors visited University of Kansas in August, 2011. They are grateful for the hospitality.}

\end{document}